\documentclass[11pt,a4paper]{amsart}
\usepackage{amssymb,color}
\usepackage[english]{babel}
\usepackage[utf8]{inputenc}
\usepackage{verbatim}
\usepackage[T1]{fontenc}             
\usepackage{amssymb,amsthm,amsmath,amsfonts}
\usepackage{floatflt,graphicx}
\usepackage{color,xcolor}

\usepackage{enumerate}
\usepackage{amsmath}
\usepackage{amsthm}

\newcounter{minutes}
\divide\time by 60
\newcounter{hours}
\multiply\time by 60 \addtocounter{minutes}{-\time}
\title{Ptolemy constant and uniformity}

\author{Eero Harmaala}
\address{Department of Mathematics and Statistics\\
  University of Turku\\ Turku, Finland}
\email{eero.harmaala@gmail.com}

\author{Riku Kl\'en }
\address{Turku PET Centre\\
  University of Turku\\ Turku, Finland}
\email{riku.klen@utu.fi}

\keywords{Ptolemy constant, uniformity constant, uniform domain}\subjclass[2010]{51M05, 30F45}

\theoremstyle{plain}
\newtheorem{thm}[equation]{Theorem}
\newtheorem{cor}[equation]{Corollary}
\newtheorem{lem}[equation]{Lemma}
\newtheorem{example}[equation]{Example}
\newtheorem{prop}[equation]{Proposition}

\theoremstyle{definition}

\theoremstyle{remark}
\newtheorem{rem}[equation]{Remark}

\numberwithin{equation}{section}

\newcommand{\B}{\mathbb{B}^2}
\newcommand{\R}{\mathbb{R}^2}

\newcommand{\UH}{\mathbb{H}^2}

\newcommand{\C}{\mathbb{C}}

\newcommand{\Hn}{ {\mathbb{H}^n} }
\newcommand{\Bn}{ {\mathbb{B}^n} }
\newcommand{\Rn}{ {\mathbb{R}^n} }

\newcommand{\ang}{{ \rm ang}\,}

\renewcommand{\Im}{{ \rm Im}\,}
\renewcommand{\Re}{{ \rm Re}\,}

\newcommand{\arcsinh}{\,\textnormal{arsinh}\,}

\makeatletter
\newcommand{\addresseshere}{%
  \enddoc@text\let\enddoc@text\relax
}
\makeatother

\begin{document}


\begin{abstract} 
  We study the Ptolemy constant and the uniformity constant in various plane domains including triangles, quadrilaterals and ellipses. To obtain our results, we use M\"obius transformations and the quasihyperbolic metric.
\end{abstract}

\maketitle

\section{Introduction}

The classical Ptolemy theorem is a relation between the four sides and two diagonals of a cyclic quadrilateral. In \cite[10.9.2]{Ber09} it is formulated as Ptolemy inequality:
\begin{prop}
  Let $ABCD$ be a quadrilateral. Then
  \[
    AB \cdot CD + AD \cdot BC \ge AC \cdot BD
  \]
  and inequality holds as equality if and only if the points $A$, $B$, $C$ and $D$ lie on a circle or on a line.
\end{prop}
Based on this fact we define Ptolemy constant, which can be used to measure roundness of plane curves.

Let $J \subset \overline{\mathbb{R}^2}$ be a Jordan curve. For points $a,b,c,d \in J$ in this order we define
\[
  p(a,b,c,d) = \frac{|a-b||c-d|+|a-d||b-c|}{|a-c||b-d|}.
\]
If one of the points $a$, $b$, $c$ or $d$ is $\infty$, then $p(a,b,c,d)$ is considered as a limit. Note that $p$ is invariant under M\"obius transformations \cite{Vuo88}.

Let $D \subset \overline{\mathbb{R}^2}$ be a domain, whose $\partial D$ is a Jordan curve. We define the \emph{Ptolemy constant} as
\begin{equation}\label{ptolemys constant}
  P(D) = \sup_{a,b,c,d \in \partial D} p(a,b,c,d),
\end{equation}
where point $a$, $b$, $c$ and $d$ occur in this order when traversing the Jordan curve in positive direction. 

One motivation for our study of the Ptolemy constant arises from a result due to L.V. Ahlfors \cite{Ahl63}, which was later reformulated by S. Rickman \cite{Ric66} as follows:

\emph{A Jordan curve $J \subset \overline{\C}$ is a quasicircle iff $\sup p(a,b,c,d)$ exists and is finite for ordered points $a,b,c,d \in J$.}

For generalisations of the Ptolemy constant to normed spaces see \cite{PinReiSha01,Zuo12}. The Ptolemy theorem has also been considered in the spherical and the hyperbolic geometries, see \cite{Val70,Val70B}.

As far as we know, explicit formulas for the Ptolemy constant for specific plane domains have not been studied in the literature before the unpublished licentiate thesis of P. Seittenranta \cite{Sei96} in 1996.

In this article we study the Ptolemy constant and try to find a connection between the Ptolemy constant and the uniformity constant, which we introduce next.

Let $G \subsetneq \Rn$ be a domain. We define the \emph{quasihyperbolic length} of a rectifiable curve $\gamma \subset G$ by
\[
  \ell_k(\gamma) = \int_\gamma \frac{|dx|}{d_G(x)},
\]
where $d_G(x) = d(x,\partial G)$. For $x,y \in G$ we define the \emph{quasihyperbolic distance} (also called the \emph{quasihyperbolic metric}) by
\[
  k_G(x,y) = \inf \ell_k (\gamma),
\]
where the infimum is taken over all rectifiable curves joining $x$ and $y$ in $G$.

For $x,y \in G$ we define the \emph{distance ratio metric} by
\[
  j_G(x,y) = \log \left( 1+\frac{|x-y|}{\min \{ d_G(x),d_G(y) \}} \right).
\]
We call the domain $G$ \emph{uniform}, if there exists a constant $A$ such that
\[
  k_G(x,y) \le A j_G(x,y)
\]
for all $x,y \in G$. The \emph{uniformity constant} is defined by
\[
  A_G = \inf \{ A \ge 1 \colon k_G(x,y) \le A j_G(x,y) \textrm{ for all } x,y \in G \}.
\]
F.W. Gehring and B.G. Osgood proved that uniform domains are invariant under quasiconformal mappings \cite{Geh79}. Another interesting property of uniformity is that it is preserved under bilipschitz mappings.
\begin{lem}\cite[Exercise 3.17]{Vuo88}\label{lem:Vuo88 lemma}
  Let $f \colon \Rn \to \Rn$ be an $L$-bilipschitz, that is
  \[
    \frac{|x-y|}{L} \le |f(x)-f(y)| \le L |x-y|
  \]
  for all $x,y \in \Rn$. If $G \subset \Rn$ is uniform, then
  $f(G)$ is uniform and $A_{f(G)} \le L^4 A_{G}$.
\end{lem}

One of the leading ideas behind our study was to establish a connection between the Ptolemy constant and the uniformity constant. These constants satisfy equality $A_G = 1 + P(G)$ in the unit ball $\Bn$, the upper half space $\Hn$ and the angular domain $S_\alpha$ for $\alpha \in (0,\pi]$. Based on our study it is clear that equality is not true in all domains, but we could not find a clear connection between the two quantities. However, we can pose the following conjecture: for any domain $D \subset \R$ whose boundary $\partial G$ is a Jordan curve, we have $A_G \ge 1+ P(G)$.

Note that when considering the Ptolemy constant it is essential to consider only domains whose boundary is a Jordan curve. If for example the boundary curve is not closed, it is easy to see that there is no connection between the Ptolemy constant and the uniformity constant, see Example \ref{example:no connection}.

One of the main results in P. Seittenranta's thesis \cite{Sei96} is the following proposition:

\begin{prop}\label{main1}
  For a triangle $T$ with the smallest angle $\alpha$
  \[
    P(T) = \frac{1}{\sin \frac{\alpha}{2}} \ge 2. 
  \]
\end{prop}

We continue this study and find a new proof for Proposition \ref{main1}. This article is based on \cite{Har12} and our main results consists of six theorems. Theorems~\ref{thm:doubleangluardomain}, \ref{thm:P(G) for parallelogram} and \ref{thm:P(G) for ellipse} consider the Ptolemy constant, and Theorems~\ref{UC in triangle}, \ref{thm:A for convex polygon} and \ref{thm:UC for ellipse} consider the uniformity constant.

\section{Preliminary results}

In this section we introduce preliminary results. Proofs of these results are provided in Appendix.

For $\alpha \in (0,2\pi)$ we denote angular domain by
\[
  S_\alpha = \{ r e^{it} \colon r>0, \, t \in (0,\alpha) \}.
\]
and for $\alpha,\beta \in (0,\pi)$ with $\alpha+\beta \ge \pi$ we denote double angular domain by
\begin{equation}\label{eqn:double angular domain}
  S_{\alpha,\beta} = S_\alpha \cap \{ 1+r e^{it} \colon r>0, \, t \in (\pi-\beta,\pi) \}.
\end{equation}

The following proposition gives the circumcenter of a triangle in complex number notation. The next results is well known, but since we have not been able to find suitable reference, we give a proof for completeness. A similar type of result is given in \cite[p. 85]{AndAnd06}.

\begin{prop}\label{prop:circumcenter}
  Let $a,b,c \in \C$ be vertices of a triangle $T$. The circumcenter of $T$ is
  \[
    \frac{(b-c)|a|^2+(c-a)|b|^2+(a-b)|c|^2}{(b-c)\bar{a}+(c-a)\bar{b}+(a-b)\bar{c}}.
  \]
\end{prop}
%
%
\begin{lem}\label{lem:circle_outer_angle}
  Let $a, b, c, d \in \C$ be distinct points forming a convex polygon $abcd$ and $\alpha$, $\beta$, $\gamma$ and $\delta$ be the angles of the polygon, respectively. Then the outer angle between circles $\mathcal{C}_{A} = \mathcal{C}(a,b,d)$ and $\mathcal{C}_{C} = \mathcal{C}(b,c,d)$ is equal to $\min \{ \alpha+\gamma,\beta+\delta \}$. Also the outer angle between circles $\mathcal{C}_{B} = \mathcal{C}(a,b,c)$ and $\mathcal{C}_{D} = \mathcal{C}(c,d,a)$ is equal to $\min \{ \alpha+\gamma,\beta+\delta \}$.
\end{lem}
%
%
%
%

\begin{lem}\label{lem:convex_polygon_Stheta}
  Let $a, b, c, d \in \C$ be distinct points forming a convex polygon $abcd$ and $\alpha$, $\beta$, $\gamma$, $\delta$ be the angles of the polygon, respectively. Then there exists a M\"obius transformation $m$ that maps the points $a, b, c, d$ in the same order to the curve $S_\theta$ with
  \[
    \theta = \min \{ \alpha + \gamma , \beta + \delta \} \hspace*{2ex} \in (0, \pi].
  \]
\end{lem}


\begin{lem}\label{lem:simplequadrilateral}
  Let $(a,b,c,d)$ be a simple quadrilateral with inner angles $\alpha$, $\beta$, $\gamma$, $\delta$. There exists a M\"obius mapping $m$, which maps the points $a$, $b$, $c$ and $d$ in this order to the curve $S_\theta$ for
  \[
    \theta = \min \{ \alpha+\gamma,\beta+\delta \} \in (0,\pi].
  \]
  Especially, $p(a,b,c,d) \le P(S_\theta)$ and mapping $m$ can be chosen so that $m(a)=0$, $m(b)=1$, $m(c)=\infty$ and $m(d)=t e^{i\theta}$, $t>0$.
\end{lem}
%
%
%
%

\begin{prop}\label{prop:quadrilateral}
  Let $(a,b,c,d)$ be a simple quadrilateral with opposite angles $\alpha$ and $\gamma$. There exists a M\"obius mapping $m$ such that $(m(a),m(b),m(c),m(d))$ is a parallelogram with $(\alpha+\gamma)/2$. Especially
  \[
    p(a,b,c,d) = p(m(a),m(b),m(c),m(d)).
  \]
\end{prop}
%

\begin{lem}\label{lem:visualanglemetricH2}
  Let $x,y \in \UH$ and $z \in \overline{\mathbb{R}}$. Then $|\measuredangle (x,z,y)|$ obtains its largest value when the circle through points $x$, $y$ and $z$ touches the real axis, and smallest value when $z$ is the intersection point of the real axis and the line through the points $x$ and $y$.
\end{lem}

\begin{lem}\label{lem:kompleksilukutulos}
  If $x > 0$, $y > 0$ and $c \in (0,1)$, then
  \[
    \arg (x+icy) > c \cdot \arg (x+iy).
  \]
\end{lem}
%
%

\begin{lem}\label{lem:apulause_4.12}
  Let $E$ be the domain enclosed by the ellipse $\partial E = \{ (x_0,y_0) \in \R \colon (x_0/a)^2+(y_0/b)^2 = 1 \}$ and $b \le a$. If $z=(t,0) \in E$ and $|t|\le a-b^2/a$, then
  \[
    d(z,\partial E) = b \sqrt{1-\frac{t^2}{a^2-b^2}}
  \]
  and the closest points to $z$ in $\partial E$ are
  \[
    \left( \frac{t}{1-b^2/a^2} , \pm b\sqrt{ \left( \frac{at}{a^2-b^2} \right)^2 -1 } \right).
  \]
  If $z=(t,0) \in E$ and $a-b^2/a < |t| < a$, then $d(z,\partial E) = a-|t|$ and the closest point to $z$ in $\partial E$ is $(at/|t|,0)$.
\end{lem}
%
%

\section{Angular domain and triangle}

We begin by considering the angular domain. We prove first the result in the special case that one of the points in the supremum of \eqref{ptolemys constant} is origin. Since $p$ is invariant under M\"obius transformations it makes no difference which one of the point we choose to be origin and thus we let $b=0$.

\begin{lem}\label{lem:sector1}
  For $\alpha \in (0,\pi]$
  \[
    \sup_{t>0,\, d>0,\, c\in(0,d)} p(t e^{i\alpha},0,c,d) \ge \frac{1}{\sin \frac{\alpha}{2}}. 
  \]  
\end{lem}
\begin{proof}
  Let $a=ce^{i \alpha}$, $c \in (0,1)$ and $d=1$. Now
  \[
    p(c e^{i\alpha},0,c,1) = \frac{|c-1|+|ce^{i\alpha}-1|}{|e^{i\alpha}-1|}
  \]
  and since $|ce^{i\alpha}-1| > |c-1|$ and $|e^{i\alpha}-1| = 2 \sin \frac{\alpha}{2}$ we obtain
  \[
    p(c e^{i\alpha},0,c,1) > \frac{2|c-1|}{e^{i\alpha}-1} = \frac{1-c}{\sin \frac{\alpha}{2}}.
  \]
  We choose $c=\varepsilon \sin \frac{\alpha}{2}$ for $\varepsilon >0$ and then
  \[
    \sup_{t>0,\, d>0,\, c\in(0,d)} p(t e^{i\alpha},0,c,d) = \sup_{c >0} p(c e^{i\alpha},0,c,1) \ge \frac{1}{\sin \frac{\alpha}{2}}-\varepsilon.
  \]
  The assertion follows as we let $\varepsilon \to 0$.
\end{proof}

\begin{rem}
  Note that Lemma \ref{lem:sector1} includes also the case $d=\infty$ as $p$ is invariant under M\"obius transformations.
\end{rem}

\begin{lem}\label{lem:sector2}
  For $\alpha \in (0,\pi]$ and $t,c > 0$ we have $p(t e^{i \alpha},0,c,\infty) \le 1/\sin \frac{\alpha}{2}$.
\end{lem}
\begin{proof}
  By the law of cosines we obtain
  \[
    p(te^{i\alpha},0,c,\infty) = \frac{t+c}{te^{i\alpha}-c} = \frac{t+c}{\sqrt{t^2+c^2-2tc \cos \alpha}} =:g(t).
  \]
  We easily obtain
  \[
    g'(t) = (c-t)\frac{c(1+\cos \alpha)}{(t^2+c^2-2tc \cos \alpha)^{3/2}}
  \]
  and thus the function $g$ obtains its maximum at $t=c$ and
  \[
    p(te^{i\alpha},0,c,\infty) \le g(c) = \frac{2}{|e^{i\alpha}-1|} = \frac{1}{\sin \frac{\alpha}{2}}. \qedhere
  \]
\end{proof}

\begin{prop}\label{prop:sectorb=0}
  For $\alpha \in (0,\pi]$
  \[
    \max \left\{ \sup_{t>0,\, d>0,\, c\in(0,d)} p(t e^{i\alpha},0,c,d) , \sup_{t>0,\, c>0} p(t e^{i\alpha},0,c,\infty) \right\} = \frac{1}{\sin \frac{\alpha}{2}}.
  \]
\end{prop}
\begin{proof}
  By Lemmas \ref{lem:sector1} and \ref{lem:sector2} we need to show that
  \[
    \sup_{t>0,\, d>0,\, c\in(0,d)} p(t e^{i\alpha},0,c,d) \le \frac{1}{\sin \frac{\alpha}{2}}. 
  \] 
  Let us denote $a=t e^{i\alpha}$ and consider a M\"obius transformation $m$ that fixes origin, maps $c$ onto positive real line and $d$ to $\infty$. Now by Lemma \ref{lem:sector2}
  \[
    p(m(a),0,m(c),m(d)) \le \frac{1}{\sin \frac{\measuredangle(m(a),0,m(c))}{2}} \le \frac{1}{\sin \frac{\measuredangle(a,0,c)}{2}} = \frac{1}{\sin \frac{\alpha}{2}},
  \]
  where the second inequality follows from the facts that $\measuredangle(m(a),0,m(c)) > \measuredangle(a,0,c)$ and the function $f(x)=1/\sin \tfrac{\alpha}{2}$ is decreasing on $(0,\pi]$.
\end{proof}

Next we consider the case where one of the points in the supremum of \eqref{ptolemys constant} is on one of the sides of the angular domain and the three other points are on the other side.

\begin{cor}\label{cor:sector1and3}
  Let $\alpha \in (0,\pi]$ and $a,b,c,d \in \partial S_\alpha$ be such points that one of them is on one side of $S_\alpha$ and the other three are on the other side. Then
  \[
    p(a,b,c,d) \le \frac{1}{\sin \frac{\alpha}{2}}.
  \]
\end{cor}
\begin{proof}
  We may assume that the points $b$, $c$ and $d$ are on the positive real axis and $b < c < d$. Denote $\beta = \measuredangle(a,b,c)$. Now $\beta \ge \alpha$ and points $a$, $b$, $c$ and $d$ are on the boundary of an angluar domain with angle $\beta$. By Proposition \ref{prop:sectorb=0} we have
  \[
    p(a,b,c,d) \le 1/\sin \tfrac{\beta}{2} \le 1/\sin \tfrac{\alpha}{2},
  \]
  where the second inequality follows as the function $f(x)=1/\sin \tfrac{\alpha}{2}$ is decreasing on $(0,\pi]$.
\end{proof}

Next we consider the angular domain in the case when there are exactly two points on each sides of the domain.

\begin{prop}\label{prop:sector2and2}
  Let $\alpha \in (0,\pi]$ and $a,b,c,d \in \partial S_\alpha$ be such points that two of them is on one side of $S_\alpha$ and the other two are on the other side. Then
  \[
    p(a,b,c,d) \le \frac{1}{\sin \frac{\alpha}{2}}.
  \]
\end{prop}
\begin{proof}
  We may assume $|b| \le |c|$. Let the circle through points $b$, $c$ and $d$ be $C$. Denote the intersection of the real axis and the tangent of $C$ at the point $b$ by $u$.
  
  We prove first that the angle $\gamma = \measuredangle(a,b,u) > \alpha$. By Proposition \ref{prop:circumcenter} the circumcenter of $C$ is
  \[
    k = b-\frac{(d-b)(c-b)}{b-\overline{b}}.
  \]
  By a straigthforward computation we obtain
  \[
    u = b+\frac{(d-b)(c-b)}{c+d-(b-\overline{b})} = \frac{cd-|b|^2}{c+d-2 \Re b}
  \]
  and
  \[
    \Re u - \Re b = \frac{\Re(d-b)\Re(c-b)-\Im(b)^2}{\Re(c+b)+\Re(d-b)} < \Re(c-b)
  \]
  implying $u<c$. On the other hand, $|b| \le c<d$ implies $u>0$ and thus $0<u<c$. If we denote $\beta = \measuredangle(0,u,b)$, then $\pi-(\alpha+\beta) = \pi-\gamma$ and thus $\gamma = \alpha+\beta \in (\alpha,\pi]$.
  
  If $a$ is contained inside $C$ then the points $a$, $b$, $c$ and $d$ can be mapped with a M\"obius transformation to an angular domain with angle $\gamma$ and angular point at $b$. By Proposition \ref{prop:sectorb=0} we obtain $p(a,b,c,d) = 1/\sin \tfrac{\gamma}{2} \le 1/\sin \tfrac{\alpha}{2}$.
  
  If $a$ is not contained inside $C$ then we consider circle $C'$ through points $c$, $d$ and $a$. The line through origin and $a$ intersects $C'$ at $a$ and $b'$. Since $b$ is inside $C'$ we have $|b'| \le |b| \le c$. We denote the angle between the line through origin and $a$, and the tangent of $C'$ at $a$ by $\gamma'$. Similarly as we obtained $\gamma > \alpha$, we now have $\gamma' > \alpha$ and by a M\"obius transformation and Proposition \ref{prop:sectorb=0} as above we collect  $p(a,b,c,d) = 1/\sin \tfrac{\gamma}{2} \le 1/\sin \tfrac{\alpha}{2}$.
\end{proof}

By combining Proposition~\ref{prop:sectorb=0}, Corollary~\ref{cor:sector1and3} and Proposition~\ref{prop:sector2and2} we obtain the Ptolemy constant in angular domain:

\begin{prop}\label{prop:angulardomain}
  For $\alpha \in (0,\pi]$
  \[
    P(S_\alpha) = \frac{1}{\sin \frac{\alpha}{2}}. 
  \]
\end{prop}

The result for the angular domain can easily be generalized for double angular domains.

\begin{thm}\label{thm:doubleangluardomain}
  Let $\alpha,\beta \in (0,\pi)$ with $\alpha+\beta \ge \pi$ and $S_{\alpha,\beta}$ the double angular domain (see \eqref{eqn:double angular domain}). Then
  \[
    P(S_{\alpha,\beta}) = \frac{1}{\sin \frac{ \min \{ \alpha,\beta,\alpha+\beta-\pi \} }{2}}. 
  \]
\end{thm}
\begin{proof}
  Boundary $\partial S_{\alpha,\beta}$ consists of a line segment $s$ and two half-lines $t$ and $u$. Let us denote the angular domain that contains $s$ and $t$ on its boundary by $S_\alpha$, the angular domain that contains $s$ and $u$ on its boundary by $S_\beta$ and the angular domain that contains $t$ and $u$ on its boundary by $S_\gamma$. Note that here $\gamma = \alpha+\beta-\pi$ and for each angular domain $S_j$, the subindex $j$ describes the size of the angle.
  
  By considering domains $S_\alpha$ and $S_\beta$ it is clear that by Proposition \ref{prop:angulardomain}
  \[
    P(S_{\alpha,\beta}) \ge \frac{1}{\sin \frac{ \min \{ \alpha,\beta \} }{2}}. 
  \]
  If we map the angular point of $S_\gamma$ to $\infty$ with a M\"obius transformation $m$, then $S_{\alpha,\beta}$ maps to a bounded domain with boundary consisting two line segments and a circular arc. As the angle between the line segments is $\gamma$ we obtain
  \[
    P(S_{\alpha,\beta}) \ge \frac{1}{\sin \frac{ \gamma }{2}}. 
  \]
  
  Let us prove that  
  \[
    P(S_{\alpha,\beta}) \le \frac{1}{\sin \frac{ \min \{ \alpha,\beta,\alpha+\beta-\pi \} }{2}}. 
  \]
  If $s$, $t$ or $u$ does not contain any of the points $a$, $b$, $c$ or $d$, then the points are contained on the boundary of $S_\alpha$, $S_\beta$ or $S_\gamma$ and the assertion follows from Proposition \ref{prop:angulardomain}. Now each of $s$, $t$ and $u$ contains at least one point and if we consider the angular domain $S_{\gamma'}$ that contains all of the points $a$, $b$, $c$ or $d$ on its boundary, we see that $\gamma' \ge \gamma$. Again the assertion follows from Proposition \ref{prop:angulardomain}
\end{proof}

We finally extend the above results to triangles by the following lemma.

\begin{lem}\label{lem:kolmio_ympyra}
  Let $T$ be a triangle with angles $\alpha$, $\beta$, $\gamma$ and $a,b,c,d \in \partial T$. Then the points $a$, $b$, $c$ and $d$ can be mapped in same order with a M\"obius transformation to $\partial S_\theta$, where $\theta \ge \min \{ \alpha,\beta,\gamma \}$ and $\theta \le \pi$.
\end{lem}
\begin{proof}
  Let $a$, $b$, $c$ and $d$ be on $\partial T$. If the points lie on two sides of the triangle then the assertion follows from Theorem \ref{thm:doubleangluardomain}. Thus we assume that at least one of the points $a$, $b$, $c$ and $d$ is on each side of the triangle.
  
  We may assume the triangle $T$ to have vertices 0, 1 and $A$. We denote the angles $\alpha$, $\beta$ and $\gamma$ as in Figure \ref{fig:kolmio-ymp} and we may assume that $b,c \in [0,1]$ and $a$, $b$, $c$ and $d$ are located counterclockwise.
  \begin{figure}[!ht]
    \includegraphics[width=.8\linewidth]{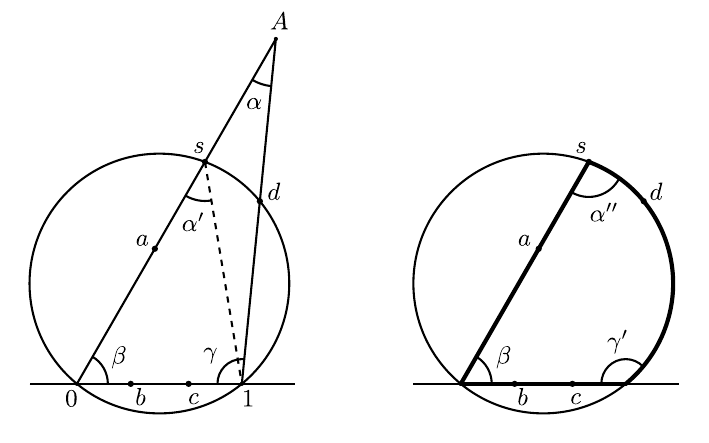}
    \caption{Proof of Lemma \ref{lem:kolmio_ympyra}.}\label{fig:kolmio-ymp}
  \end{figure}
  
  Now the circle through $a$, $0$ and $1$ contains $d$ or the circle through $0$, $1$ and $d$ contains $a$. We may assume that the circle $C$ through $0$, $1$ and $d$ encircles $a$ as the other case is symmetric.
  
  We denote ${s} = (0,A) \cap \partial C$ and note that $a \in (0,s]$. Since $A$ is outside $C$ we have $\alpha' = \measuredangle(0,s,1) > \alpha$. Denote the angle between $C$ and $[0,A]$ at $s$ by $\alpha''$ and the angle between $C$ and $[0,1]$ at $1$ by $\gamma'$. Now $\alpha''=\alpha'+\beta > \alpha$ and $\gamma' > \gamma$. By a M\"obius transformation that takes $(0,1,s)$ to $(\infty,0,1)$ the points $a$, $b$, $c$ and $d$ are mapped in this order to $\partial S_{\gamma',\alpha''}$ and the assertion follows from Theorem \ref{thm:doubleangluardomain}.
\end{proof}

Finally, we can give a proof for Proposition \ref{main1}. Let $a,b,c,d \in \partial T$. By Lemma \ref{lem:kolmio_ympyra} we can map $a$, $b$, $c$ and $d$ in this order by a M\"obius transformation to $\partial S_\theta$, where $\theta = \min \{ \alpha,\beta,\alpha+\beta-\pi \}$, and Proposition \ref{main1} follows from Proposition \ref{prop:angulardomain}.

\section{Other domains}

We consider the Ptolemy constant for quadrilaterals, ellipses and convex plane domains. We begin with quadrilaterals.

\begin{prop}\label{prop:convexquadrilateral}
  Let $(a,b,c,d)$ be a convex quadrilateral with two opposite angles $\alpha$ and $\gamma$. Then
  \[
    p(a,b,c,d) \le \frac{1}{\sin \frac{\alpha+\gamma}{2}}.
  \]
\end{prop}
\begin{proof}
  If we denote the two other angles of $(a,b,c,d)$ by $\beta$ and $\delta$, then
  \[
    \sin\frac{\beta+\delta}{2} = \sin \frac{\alpha+\gamma}{2}
  \]
  and the assertion follows from Proposition~\ref{prop:angulardomain} and Proposition~\ref{prop:quadrilateral}.
\end{proof}

\begin{lem}\label{lem:convexquadrilateral}
  Let $(x,a,b,y)$ be a convex quadrilateral and $z$ be a point on the polyline $xaby$. Then the angle $\gamma = \measuredangle(yzx)$ obtains its smallest value at $z=a$ or $z=b$.
\end{lem}
\begin{proof}
  Let us first assume that $z \in [a,b]$. Denote the line through points $a$ and $b$ by $\ell(a,b)$, and the line through points $x$ and $y$ by $\ell(x,y)$. Denote the intersection of $\ell(a,b)$ and $\ell(x,y)$ by $k$. Since $(x,a,b,y)$ is convex $k \notin [x,y]$. By Lemma \ref{lem:visualanglemetricH2}, $\gamma$ obtains its minimal value at $k$ and it is clear that $\gamma$ increases as $z$ is moved further away from $k$ along the line $\ell(a,b)$, see Figure \ref{fig:kmax_knelik}. Thus the assertion follows.
  \begin{figure}[!ht]
    \includegraphics[width=.9 \linewidth]{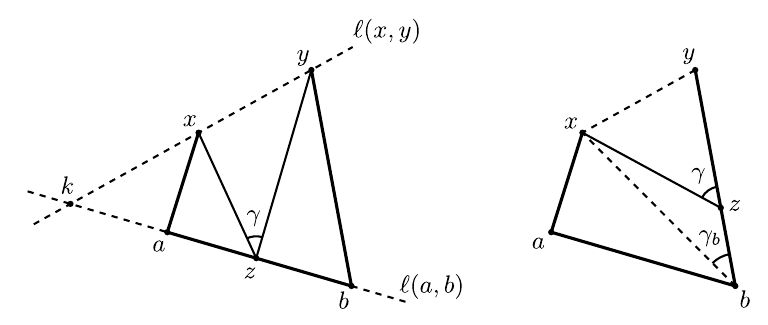}
    \caption{Proof of Lemma \ref{lem:convexquadrilateral}. The cases $z \in [a,b]$ (on left) and $z \in [b,y]$ (on right).}
    \label{fig:kmax_knelik}
  \end{figure}
  
  If $z \in [a,x]$ or $z \in [b,y]$ the assertion is clear, see Figure \ref{fig:kmax_knelik}.
\end{proof}

\begin{cor}\label{cor:convexpolygon}
  Let $(x,a_1,a_2,\dots ,a_n,y)$, $n \ge 1$, be a convex polygon and $z$ be a point on the polyline $x a_1 a_2\cdots a_ny$. Then the angle $\gamma = \measuredangle(yzx)$ obtains its smallest value at $z=a_i$ for some $i \in \{ 1,2,\dots ,n \}$.
\end{cor}

Our next two results give lower and uppers for the Ptolemy constant in a parallelogram.

\begin{prop}\label{prop:parallellogram_lower}
  Let $G$ be a parallelogram with smallest angle $\alpha$ and sides $r$ and $s$. Then
  \[
    P(G) \ge \max \left\{ \frac{1}{\sin \frac{\alpha}{2}} , \sqrt{1+ \left( \frac{\max \{ r,s \} }{2 \min \{ r,s \} \sin \alpha} \right)^2 }, \frac{1}{2} \left( f(r,s,\alpha) + \frac{1}{f(r,s,\alpha)} \right) \right\},
  \]
  where
  \[
    f(r,s,\alpha) = \frac{\sqrt{ r^2+2rs \cos \alpha + s^2 }}{\min \{r,s\} \sin \alpha}.
  \]
\end{prop}
\begin{proof}
  By Proposition \ref{prop:angulardomain} it is clear that
  \[
    P(G) \ge \frac{1}{\sin \frac{\alpha}{2}}.
  \]
  
  We may assume $r \ge s$. Let $A$, $B$, $C$ and $D$ be the vertices of $G$ and let $a,b,c,d \in \partial G$.
  
  We prove first that
  \[
    P(G) \ge \frac{1}{2} \left( f(r,s,\alpha) + \frac{1}{f(r,s,\alpha)} \right).
  \]
  Since $r \ge s$, we have $(r+s\cos \alpha)/2 \ge s \cos \alpha$. We choose $b=B$, $d=D$ and points $a$ and $c$ in a way that they divide the whole length of $G$, that is $r+s \cos \alpha$, into half (see left-hand side of Figure \ref{fig:ksuunnikas}).
  \begin{figure}[!ht]
    \includegraphics[width=\linewidth]{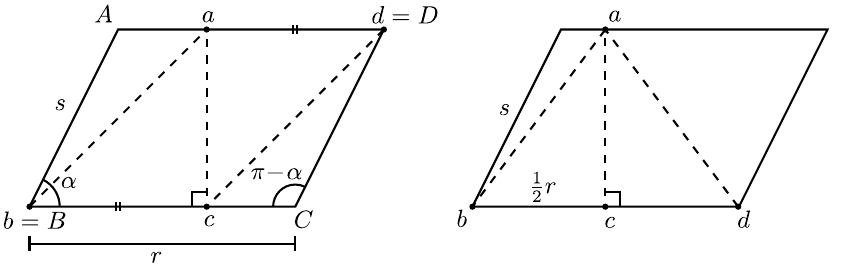}
    \caption{Proof of Proposition \ref{prop:parallellogram_lower}.}\label{fig:ksuunnikas}
  \end{figure}
  
  Now $|a-d|=|b-c|=(r+s \cos \alpha)/2$, $|a-c|=s \sin \alpha$,
  $|a-b|=|c-d|=\sqrt{s^2 \sin^2 \alpha +|a-d|^2}$ and $|b-d|=\sqrt{r^2+s^2+2rs\cos\alpha}$. Now
  \begin{eqnarray*}
    p(a,b,c,d) & = & \frac{1}{2}\frac{(r^2+2rs\cos\alpha+s^2\cos^2\alpha)+2s^2\sin^2\alpha}{s \sqrt{r^2+s^2+2rs\cos\alpha} \sin\alpha}\\
    & = & \frac{1}{2} \left( f(r,s,\alpha) + \frac{1}{f(r,s,\alpha)} \right).
  \end{eqnarray*}
  
  Finally, we show that
  \[
    P(G) \ge \sqrt{1+ \left( \frac{r}{2 r \sin \alpha} \right)^2 }.
  \]
  If $r < 2s \cos \alpha$, then
  \[
    \sqrt{1+\left( \frac{r}{2s \sin \alpha} \right)^2} < \sqrt{\frac{1}{1+\tan^2 \alpha}}=\frac{1}{\sin \alpha}
  \]
  and thus
  \[
    P(G) > \frac{1}{2} \frac{\sqrt{r^2+2rs \cos \alpha+s^2}}{s \sin \alpha} > \frac{1}{2} \frac{\sqrt{3r^2+s^2}}{s \sin \alpha} \ge \frac{1}{\sin \alpha} > \sqrt{1+\left( \frac{r}{2s \sin \alpha} \right)^2}.
  \]
  We assume $r \ge 2s \cos \alpha$. We choose $b=A$, $c=(A+B)/2$, $d=C$ and $a$ is the intersection point of the side $[A,D]$ and the perpendicular bisector of $[B,C]$ (see right-hand side of Figure \ref{fig:ksuunnikas}). Now
  \[
    p(a,b,c,d) = \sqrt{1+ \left( \frac{r}{2 r \sin \alpha} \right)^2 }
  \]
  and the assertion follows.
\end{proof}

\begin{prop}\label{prop:Pforparallelogram}
  Let $G$ be a parallelogram with smallest angle $\omega$ and sides $r$ and $s$. Then
  \[
    P(G) \le \frac{\sqrt{ r^2+2rs \cos \omega + s^2 }}{\min \{r,s\} \sin \omega}.
  \]
\end{prop}
\begin{proof}
  We may assume $s \le r$. Let $a,b,c,d \in \partial G$ be points in this order. We denote the inner angles of the parallelogram $(a,b,c,d)$ by $\alpha$, $\beta$, $\gamma$ and $\delta$, respectively. We may assume $\alpha + \gamma \ge \pi$.
  
  If $a$ and $c$ lie on the same side of $G$ so does $b$ and $d$, because $\beta + \delta\le \pi$. In this case $p(a,b,c,d)=1$.
  
  Let us assume that $a$ and $c$ lie on adjacent sides of $G$. Now $\beta+\delta \ge \omega$ and by Proposition \ref{prop:convexquadrilateral}
  \begin{eqnarray*}
    p(a,b,c,d) & \le & \frac{1}{\sin \frac{\beta+\delta}{2}} \le \frac{1}{\sin \frac{\omega}{2}} = \frac{2 \cos \frac{\omega}{2}}{\sin \omega} \\
    & \le & \frac{\sqrt{2+2\cos \omega}}{\sin \omega} \le \frac{\sqrt{ r^2+2rs \cos \omega + s^2 }}{\min \{r,s\} \sin \omega}.
  \end{eqnarray*}
  
  Let us finally assume that $a$ and $c$ lie on opposite sides of $G$. By Corollary \ref{cor:convexpolygon} we may assume that $b$ and $d$ are vertices of $G$. As above we know by Proposition \ref{prop:convexquadrilateral} that
  \begin{equation}\label{eqn:convexpolygon}
    p(a,b,c,d) \le \frac{1}{\sin \frac{\beta+\delta}{2}}.
  \end{equation}
  
  If $b$ and $d$ are opposite vertices then $\beta, \delta \ge \theta$, where $\theta$ is the angle between the diagonal and a side of $G$. We may assume that $\theta$ is the smaller of the two possible angles. Now $(\beta+\gamma)/2 \ge \theta$ implying
  \[
    \sin \frac{\beta+\delta}{2} \ge \sin \theta = \frac{s \sin \omega}{\sqrt{(r+s \cos \omega)^2+s^2 \sin^2 \omega}} = \frac{s \sin \omega}{\sqrt{r^2+2 r s \cos \omega+ s^2}}
  \]
  and the assertion follows from \eqref{eqn:convexpolygon}.
  
  If $b$ and $d$ are adjacent vertices, then we may assume that $c \in [b,d]$. By Lemma \ref{lem:visualanglemetricH2} the largest possible value for $p(a,b,c,d)$ is attained for $a$, which is on the perpendicular bisector $p$ of points $b$ and $d$. Even if $p$ does not intersect the side of $G$ that is opposite to $[b,d]$, we can still use the estimate
  \[
    \sin \frac{\beta+\delta}{2} \ge \frac{s \sin \omega}{\sqrt{\left(  \frac{r}{2} \right)^2 +s^2 \sin^2 \omega}} \ge \frac{s \sin \omega}{\sqrt{r^2+s^2}} \ge \frac{s \sin \omega}{\sqrt{r^2+2 r s \cos \omega+ s^2}}
  \]
  and the assertion follows from \eqref{eqn:convexpolygon}.  
\end{proof}
  
  \begin{figure}[!ht]
    \includegraphics[width=\linewidth]{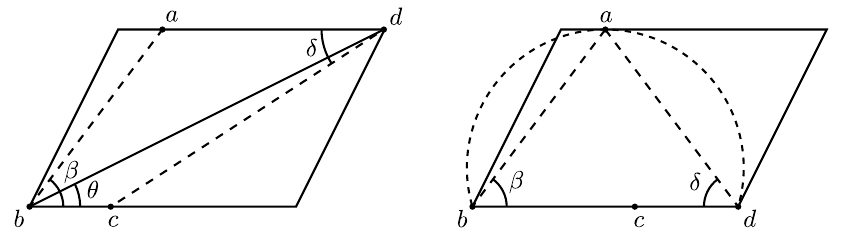}
    \caption{Proof of Proposition \ref{prop:Pforparallelogram}. The points $a$ and $c$ are on the opposite sides. The points $b$ and $d$ are on opposite vertices (on left) and on adjacent vertices (on right).}
    \label{fig:suunnikas_yla}
  \end{figure}

\begin{thm}\label{thm:P(G) for parallelogram}
  Let $G$ be a parallelogram with smallest inner angle $\alpha \in (0,\tfrac{\pi}{2}]$ and sides $r$ and $s$. Then
  \[
    \frac{1}{2} \left( A + \frac{1}{A} \right) \le P(G) \le A = \frac{\sqrt{r^2+2rs\cos \alpha+ s^2}}{\min\{ r,s \} \sin \alpha}.
  \]
\end{thm}
\begin{proof}
  Follows from Propositions \ref{prop:parallellogram_lower} and \ref{prop:Pforparallelogram}.
\end{proof}

We collect two corollaries as special cases of Theorem \ref{thm:P(G) for parallelogram}.

\begin{cor}
  If $G$ is a rhombus (a parallelogram with r=s) with smallest angle $\alpha \in (0,\tfrac{\pi}{2}]$, then
  \[
    P(G) = \frac{1}{\sin \frac{\alpha}{2}}.
  \]
\end{cor}
\begin{proof}
  By Proposition \ref{prop:Pforparallelogram}
  \[
    P(G) \le \frac{\sqrt{2+2 \cos \alpha}}{\sin \alpha} = \frac{1}{\sin \frac{\alpha}{2}}
  \]
  and by Proposition \ref{prop:parallellogram_lower}
  \[
    P(G) \ge \max \left\{ \frac{1}{\sin \frac{\alpha}{2}} , \sqrt{1+ \left( \frac{1}{2 \sin \alpha} \right)^2 }, \frac{1}{2} \left( \sin \frac{\alpha}{2} + \frac{1}{\sin \frac{\alpha}{2}} \right) \right\},
  \]
  so the assertion follows.
\end{proof}

\begin{cor}
  If $R$ is a rectangle with sides $r$ and $s$, then
    \[
      \max \left\{ \sqrt{2},\sqrt{1+\frac{\max \{ r,s \}^2 }{4 \min \{ r,s \}^2 }} \right\} \le P(R) \le \sqrt{1+\frac{\max \{ r,s \}^2 }{\min \{ r,s \}^2 }}.
    \]
\end{cor}
\begin{proof}
  Follows from Propositions \ref{prop:parallellogram_lower} and \ref{prop:Pforparallelogram}.
\end{proof}

Our final goal is to study the Ptolemy constant in ellipses. First we introduce a more general result for convex domains.

\begin{prop}\label{prop:convexcurve}
  Let $J$ be a convex curve with parametrisation $\gamma \colon [0,1] \to \C$. If $\alpha = |\ang \gamma(0)-\ang \gamma(1)| < \pi$, then
  \[
    P(J) \le \frac{1}{\sin \frac{\pi-\alpha}{2}}.
  \]
\end{prop}
\begin{proof}
  The angle between the tangents of $J$ at the points $\gamma(0)$ and $\gamma(1)$ is at least $\pi-\alpha$, see Figure \ref{fig:konv_kaari}.
  
  Let $a,b,c,d \in [0,1]$ be such that $a<b<c<d$. Let $k$ be the intersection of the tangents of $J$ at points $\gamma(a)$ and $\gamma(d)$ and $\beta$ the angle between the tangents. Now $\beta \ge \pi-\alpha$. Let us denote the angle between lines through points $\gamma(a)$, $\gamma(b)$ and $\gamma(d),k$ by $\delta \ge \beta$. Finally, we denote the angle between lines through points $\gamma(a)$, $\gamma(b)$ and $\gamma(c),\gamma(d)$ by $\varphi$. Now $\varphi \ge \delta \ge \pi-\alpha$ and the assertion follows from Proposition \ref{prop:angulardomain}.
\end{proof}
  
  \begin{figure}[!ht]
    \includegraphics[width=\linewidth]{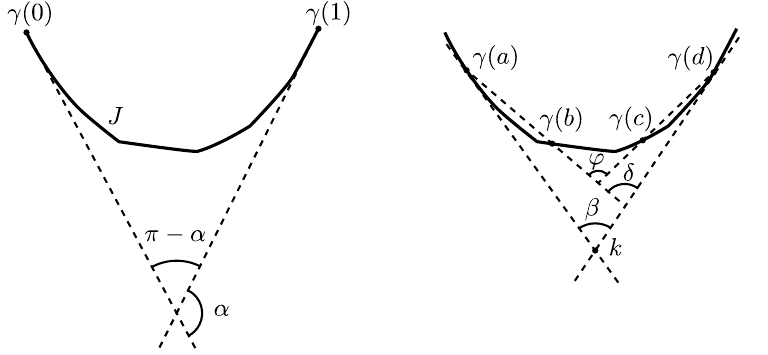}
    \caption{Proof of Proposition \ref{prop:convexcurve}.}
    \label{fig:konv_kaari}
  \end{figure}

\begin{thm}\label{thm:P(G) for ellipse}
  Let $E$ be an ellipse with semiaxis $a$ and $b$. Then
  \[
    \frac{1}{2} \left( \frac{a}{b}+\frac{b}{a} \right) \le P(E) \le \frac{1}{\sin \frac{b \pi}{2a}} \le \frac{2}{\pi} \left( \frac{a}{b}+\frac{b}{a} \right).
  \]
\end{thm}
\begin{proof}
  If $a=b$, the claim is clear. We assume $b<a$ and that the semiaxes lie on the real and the imaginary axes. Now
  \[
    p(a,bi,-a,-bi) = \frac{2 \cdot (\sqrt{a^2+b^2})^2}{2a \cdot 2b} = \frac{1}{2} \left( \frac{a}{b} + \frac{b}{a} \right) \le P(E).
  \]
  
  For the upper bound of $P(E)$ we consider scaling $E$ to circle $C$ with center at origin and radius $b$. The scaling is horizontal with scaling factor $c=b/a$. Four points on $C$ form a convex quadrilateral and we denote the angles by $\alpha$, $\beta$, $\gamma$ and $\delta$. Now $\alpha + \gamma = \pi = \beta + \delta$.
  
  Each angle has two sides and when scaling $E$ to $C$ the angle $\eta$ between a side and a horizontal line changes. Lemma~\ref{lem:kompleksilukutulos} gives a lower bound for the change of $\eta$. In the case when the scaling causes maximal decrease in $\alpha+\gamma$, we obtain for new scaled angles $\alpha'$ and $\beta'$ that $\alpha' + \beta' \in (c\pi,\pi]$. Now the upper bound for $P(E)$ follows from Lemma~\ref{lem:simplequadrilateral} and Proposition~\ref{prop:angulardomain}.

  To prove the last inequality we show that for $c \in (0,1]$,
  \[
  \frac{1}{ \sin \left( c \cdot \pi / 2 \right) } \; < \;
  \frac{4}{\pi} \cdot \frac{1}{2} \left( c + \frac{1}{c} \right),
  \]
  which is equivalent to $f(c) < 4/\pi$ for
  \[
    f(c) \; = \; \frac{1}{ \sin ( c \pi / 2 ) } \; / \; \frac{1}{2} \left( c + \frac{1}{c} \right) \; = \; \frac{2}{\sin ( c \pi / 2 ) \cdot (c + 1/c) }.
  \]
  Now
  \[
    f'(c) = \frac{-2}{(1+c^2)^2 \sin ( c \pi / 2 ) } \cdot \left( (1+c^2) \cdot (c \pi /2) \cot (c \pi /2) - (1-c^2) \right).
  \]
  By \cite{Bec78} for $x \in (0, \pi / 2]$
  \[
    x \cot x \geq 1 - \frac{4x^2}{\pi^2}
  \]
  and thus $(c \pi /2) \cot (c \pi /2) \geq (1-c^2)$ implying
  \begin{equation}\label{equ_ellipFunk}
  f'(c) \leq \frac{-2}{(1+c^2)^2 \sin ( c \pi / 2 ) } \cdot (1 - c^2) ( 1 + c^2 - 1 ) \leq 0.
\end{equation}
  Since
  \[
  \lim\limits_{c \rightarrow 0+} f(x) \; = \; \lim\limits_{c \rightarrow 0+} \frac{4}{\pi} \cdot {\left( (c^2 + 1) \sin ( c \pi / 2 ) / ( c \pi / 2 ) \right)}^{-1} \; = \; \frac{4}{\pi},
  \]
  the assertion follows.
\end{proof}

\section{Uniformity}

In this section we derive new estimates for the uniformity constant. To consider the uniformity constant we often need to estimate the quasihyperbolic distance, because explicit formula for it is known for very few simple domains. One of these is the complement of the origin. Martin and Osgood proved \cite[p. 38]{MarOsg86} that for all $x,y \in \Rn \setminus \{ 0 \}$ 
  \begin{equation}\label{martin-osgood formula}
    k_{\Rn \setminus \{ 0 \}}(x,y) = \sqrt{\measuredangle(x,0,y)^2+\left( \log \frac{|x|}{|y|} \right)^2 },
  \end{equation}
  where $\measuredangle(x,0,y)$ is the angle between line segments $[0,x]$ and $[0,y]$.
  
  In the following example we consider the uniformity constant of a circular arc. We show that in this case there is no connection between the Ptolemy constant and the uniformity constant.

\begin{example}\label{example:no connection}
  Let us consider domain $D$ in $\C$, whose boundary consists of an arc of the unit circle. Then $P(D) = 1$ and $A_D$ depends on the length of the $\partial D$ and $A_D$ increases as the length of $\partial D$ increases. For $a \in (0,\pi/2)$ we define
  \[
    \partial D = \{ z \in \C \colon z=e^{i t}, \, t \in [a,2\pi] \}.
  \]
  We derive a lower bound $l=l(a)$ for $A_D$ in terms of $a$ and show that $l(a) \to \infty$ as $a \to 0$.
  
  We fix points $x,y \in D$ to be $x=0$ and $y=2$. Now $d_D(x)=d_D(y)=1$ and $j(x,y) = \log 3$.
  
  Denote $z=e^{i ta}$ and $u=(1+z)/2$. We estimate
  \[
    k_D(x,y) \ge k_D(x,[z,1])+k_D([z,1],[1,\infty)),
  \]
  where $k_D(x,[z,1])$ denotes the quasihyperbolic distance from point $x$ to line segment $[z,1]$ and $k_D([z,1],[1,\infty))$ denotes the quasihyperbolic distance from line segment $[z,1]$ to the set $[1,\infty) = \{ (t,0) \in \C \colon t \ge 1 \}$. By \cite[Remark, 4.26]{Kle08} and \cite[Lemma 2.2]{KleHar15} we can calculate
  \begin{eqnarray*}
    k_D(x,[z,1]) = k_G(x,u) & = & \log \frac{2 \left( |u-x|+\sqrt{\frac{|z-1|^2}{4}+|u-x|^2} \right)}{|1-z|}\\
    & = & \log \frac{1+\cos \frac{a}{2}}{\sin \frac{a}{2}},
  \end{eqnarray*}
  since $|u-x| = |u| = \cos (a/2)$ and $|1-z| = 2 \sin (a/2)$. By \eqref{martin-osgood formula}
  \begin{eqnarray*}
    k_D([z,1],[1,\infty)) & \ge & k_{\C \setminus \{ 1 \}}(u,y) = \sqrt{\measuredangle(u,1,y)^2+\left( \log \frac{|u-1|}{|y-1|} \right)^2 }\\
    & \ge & \measuredangle(u,1,y) = \frac{\pi}{2}+\frac{a}{2}.
  \end{eqnarray*}
    
    Putting the estimates together we obtain
  \begin{eqnarray*}
    A_D \ge \frac{k_D(x,y)}{j_D(x,y)} \ge \frac{\log \frac{1+\cos \frac{a}{2}}{\sin \frac{a}{2}}+\frac{\pi}{2}+\frac{a}{2}}{\log 3} \to \infty
  \end{eqnarray*}
  as $a \to 0$.
\end{example}

We introduce the following exact result \eqref{eq:Linden} for the angular domain and build up results to obtain a lower bound for the uniformity constant for convex polygons.

H. Lindén proved \cite{Lin05} that for $\alpha \in (0,\pi]$,
\begin{equation}\label{eq:Linden}
  A_{S_\alpha} = 1+ \frac{1}{\sin \frac{\alpha}{2}}.
\end{equation}

\begin{prop}\label{prop:domain with an angle}
  Let $\alpha \in (0,\pi)$ and $G \subset S_\alpha$ be a domain such that for some $r>0$
  \[
    G \cap B^2(r) = S_\alpha \cap B^2(r).
  \]
  Then $A_G \ge A_{S_\alpha}$.
\end{prop}
\begin{proof}
  Since $S_\alpha$ is uniform, for any $\varepsilon > 0$ there exists points $x_0, y_0 \in S_\alpha$ such that
  \[
    k_{S_\alpha}(x_0, y_0) \geq (A_{S_\alpha} - \varepsilon) j_{S_\alpha}(x_0, y_0).
  \]
  Let us denote $R = 2 \cdot \max \{ |x_0|, |y_0| \}$. Now points $x_0$ and $y_0$ are contained in $S_\alpha \cap \mathbb{B}(0, R)$ and thus points $r/R \cdot x_0$ and $r/R \cdot y_0$ are contained in $G \cap \mathbb{B}(0,r)$. We denote that
\begin{eqnarray*}
  j_G(r/R \cdot x_0, r/R \cdot y_0) & = & \log \left( 1+\frac{|\frac{r}{R}x_0-\frac{r}{R}y_0|}{\min \{ d_{S_\alpha}(\frac{r}{R}x_0),d_{S_\alpha}(\frac{r}{R}y_0) \}} \right)\\
  & = & \log \left( 1+\frac{|x_0-y_0|}{\min \{ d_{S_\alpha}(x_0),d_{S_\alpha}(y_0) \}} \right)\\
  & = & j_{S_\alpha}(x_0, y_0).
\end{eqnarray*}

\begin{figure}[!ht]\centering
\includegraphics[scale=1]{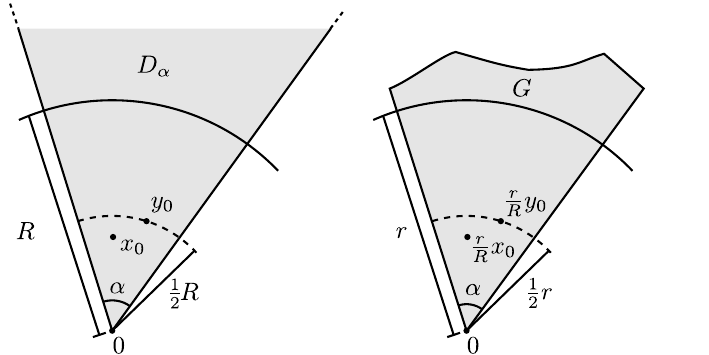}
\caption{Points $x_0$ and $y_0$ in $S_\alpha$ and their images under mapping $z \mapsto (r/R)z$.\label{kuva_kulmaSkaalaus}}
\end{figure}

  Next we show that
  \[
    k_G(r/R \cdot x_0, r/R \cdot y_0) \geq k_{S_\alpha}(x_0, y_0).
  \]
  For any point $z \in G$ or equivalently $(R/r)z \in (R/r)G \subseteq S_\alpha$ we have
  \[
    d(\partial G, z)
    \; = \;
    r/R \cdot d(\partial ((R/r)G), (R/r)z)
    \; \leq \;
    r/R \cdot d(\partial S_\alpha ,(R/r)z).
  \]
  Let $\gamma \subset G$ be a rectifiable path joining $r/R \cdot x_0$ and $r/R \cdot y_0$. Now
  \[
    k_G(r/R \cdot x_0, r/R \cdot y_0)
    \; = \;
    \inf\limits_{\gamma \subset G} \int\limits_\gamma \frac{|dz|}{d(\partial G, z)}
    \; \geq \;
    \inf\limits_{\gamma \subset G} \int\limits_\gamma \frac{|dz|}{r/R \cdot d(\partial S_\alpha ,(R/r)z)}.
  \]
  and further
  \begin{eqnarray*}
    k_G(r/R \cdot x_0, r/R \cdot y_0) & \geq & \inf\limits_{\gamma \subset G} \int\limits_\gamma \frac{|dz|}{r/R \cdot d(\partial S_\alpha ,(R/r)z)}\\
    & = & \inf\limits_{\Gamma \subset (R/r)G} \; \int\limits_{\Gamma} \frac{|dw|}{ d(\partial S_\alpha ,w)}\\
    & \geq &\inf\limits_{\Gamma' \subset S_\alpha} \; \int\limits_{\Gamma'} \frac{|dw|}{ d(\partial S_\alpha ,w)}\\
    & = & k_{S_\alpha}(x_0, y_0),
\end{eqnarray*}
because paths $\Gamma' \subset S_\alpha$ (joining $x_0$ and $y_0$) covers all the paths $\Gamma \subset (R/r)G$.

  By putting all together we obtain
  \begin{eqnarray*}
    k_G(r/R \cdot x_0, r/R \cdot y_0) & \geq & k_{S_\alpha}(x_0, y_0)\\
    & \geq & (A_{S_\alpha} - \varepsilon) j_{S_\alpha}(x_0, y_0)\\
    & = & (A_{S_\alpha} - \varepsilon) j_G(r/R \cdot x_0, r/R \cdot y_0)
  \end{eqnarray*}
  and the assertion follows as we let $\varepsilon \to 0$.
\end{proof}

\begin{thm}\label{thm:A for convex polygon}
  Let $G$ be a convex polygon with smallest inner angle $\alpha \in (0,\pi)$. Then
  \[
    A_G \ge 1+\frac{1}{\sin \frac{\alpha}{2}}.
  \]
\end{thm}
\begin{proof}
  For each vertex of the polygon $G$ we may use Proposition \ref{prop:domain with an angle} for the inner angle $\alpha_m$. Since $G$ is convex $\alpha_m \in (0,\pi]$. Now $A_G \ge \max_m A_{S_{\alpha_m}}$ and the maximum is obtained for the smallest angle $\alpha_m$.
\end{proof}

Our next goal is to find a lower bound for the uniformity constant in triangle. To obtain it we estimate the quasihyperbolic distance in angular domain and the uniformity constant in cut angular domain $S_\alpha \cap B^2(r)$.

\begin{lem}\label{lem:qh estimate in angluar domain}
  Let $\alpha \in (0,\pi]$ and $x,y \in S_\alpha$ with $|y| \le |x|$. Then
  \[
    k_{S_\alpha}(x,y) \ge \frac{1}{\sin \frac{\alpha}{2}} \ln \frac{|x|}{|y|}.
  \]
\end{lem}
\begin{proof}
  We denote $z(t) = r(t) e^{i\theta(t)}$ and thus $|dz| \geq |dr|$. Since $d(\partial S_{\alpha}, z) \leq r \sin(\alpha/2)$ we obtain
  \begin{eqnarray*}
    k_{S_{\alpha}}(x,y) & = & \inf\limits_{\gamma} \int\limits_{\gamma} \frac{|dz|}{d(\partial S_{\alpha}, z)} \; \geq \; \int\limits_{|x|}^{|y|} \frac{|dr|}{r \sin(\alpha/2)}\\
    & \geq & \frac{1}{\sin \frac{\alpha}{2}} \, \left| \, \int\limits_{|x|}^{|y|} \frac{dr}{r} \, \right| \; = \; \frac{1}{\sin \frac{\alpha}{2}} \log \frac{|y|}{|x|}
  \end{eqnarray*}
  and the assertion follows.
\end{proof}

\begin{prop}
  Let $\alpha \in (0,\pi)$ and $G \subset \R$ be a domain such that for some $r>0$
  \[
    G \cap B^2(r) = S_\alpha \cap B^2(r).
  \]
  Then $A_G \ge A_{S_\alpha}$.
\end{prop}
\begin{proof}
  The assertion can be proved in a similar way as Proposition \ref{prop:domain with an angle}.
  
  We choose $r$, $x_0$ and $y_0$ as in the proof of Proposition \ref{prop:domain with an angle}. For $R \ge 2 \max \{ |x_0|,|y_0| \}$ we obtain
  \[
    j_G \left( (r/R) x_0, (r/R) y_0 \right) = j_{S_\alpha} (x_0,y_0).
  \]
  
  We estimate next the quasihyperbolic distance between $x_0$ and $y_0$. We denote $c = \frac12 A_{S_\alpha} j_{S_\alpha} (x_0,y_0)$ and set
  \[
    R \ge 2 \max \{ |x_0|,|y_0| \} e^{c \sin (\alpha / 2)}
  \]
  or equivalently
  \[
    \frac{1}{\sin \frac{\alpha}{2}} \log \frac{\frac{1}{2}R}{\max \{ |x_0|,|y_0| \}} \ge c.
  \]
  
  Let $\gamma \subset G$ be a curve joining points $(r/R) x_0$ and $(r/R) y_0$. We denote $(R/r)G$ by $G'$ and the curve $(R/r)\gamma$ by $\gamma'$. Note that $\gamma'$ joins the points $x_0$ and $y_0$.
  
  If $\gamma' \subset B^2 (R/2)$, then for each $z \in \gamma'$ we have
  \[
    d(\partial G',z) = d(\partial S_\alpha,z)
  \]
  and thus
  \[
    k_G((r/R)x_0,(r/R)y_0) = k_{S_\alpha}((r/R x_0),(r/R)y_0).
  \]
  Now the proof continues as in the proof of Proposition \ref{prop:domain with an angle}.
  
  If $\gamma' \not \subset B^2 (R/2)$, then $\gamma'$ goes from $S^1 (\max \{ |x_0|,|y_0| \}$ to $S^1 (R/2)$ and back at least once. Since for every $z \in B^2(R/2)$, we have $D(\partial G',z) = d(\partial S_\alpha,z)$, Lemma \ref{lem:qh estimate in angluar domain} gives
  \begin{eqnarray*}
    k_{G'}(x_0,y_0) & \ge & 2 \frac{1}{\sin \frac{\alpha}{2}}\log \frac{\frac12 R}{\max \{ |x_0|,|y_0| \}} > 2c\\
    & = & A_{S_\alpha} j_{S_\alpha}(x_0,y_0) = A_{S_\alpha} j_{S_\alpha}((r/R)x_0,(r/R)y_0)
  \end{eqnarray*}
  and the assertion follows.
\end{proof}

We estimate the uniformity constant in rhombi and obtain an estimate for rectangles as a special case.

\begin{prop}\label{prop: UC for rhombus}
  If $G$ is a rhombus with smallest angle $\alpha$, then
  \[
    A_G \ge \frac{2}{\sin \frac{\alpha}{2}}.
  \]
\end{prop}
\begin{proof}
  We may choose $G$ so that its vertices are $1$, $ti$, $-1$ and $-ti$ for $t \in(0,1)$. Let $x=(s,0)$ for $s \in [0,1)$ implying $d_G(x) = (1-|x|)\sin(\alpha/2)$. The quasihyperbolic geodesic from $x$ to $-x$ is the line segment $[-x,x]$ and thus
  \[
    k_G(-x,x) = 2\int_0^{|x|} \frac{du}{(1-u)\sin \tfrac\alpha2} = \frac{-2 \log (1-|x|)}{\sin \tfrac\alpha2}.
  \]
  For the distance ratio metric we obtain
  \begin{eqnarray*}
    j_G(-x,x) & = & \log \left( 1+\frac{2|x|}{(1-|x|)\sin \tfrac\alpha2} \right)\\
    & = & \log \left( \sin \tfrac\alpha2+(2-\sin \tfrac\alpha2)|x| \right) - \log \left( (1-|x|)\sin \tfrac\alpha2 \right).
  \end{eqnarray*}
  We denote $u=|x|$ and $C=\sin \tfrac\alpha2$. The l'H\^{o}pital rule gives
  \begin{eqnarray*}
    \frac{k_G(-x,x)}{j_G(-x,x)} & \ge & \lim_{u \to 1} \frac{k_G(-x,x)}{j_G(-x,x)}\\
    & = & \lim_{u \to 1} \frac{2}{C(1-u)} \cdot \frac{1}{\frac{2-C}{C+u(2-C)} + \frac{C}{C(1-u)}}\\
    & = & \lim_{u \to 1} \frac{2}{C} \cdot \frac{C+u(2-C)}{(2-C)(1-u)+C+u(2-C)} = \frac{2}{C}
  \end{eqnarray*}
  and the assertion follows.
\end{proof}


\begin{cor}\label{UC for rectangular}
  For rectangle $R$ with sides of length $a$ and $b\le a$ the uniformity constant is
  \[
    A_R \ge 2 \sqrt{2} \left( \frac{b}{a} \right)^4.
  \]
\end{cor}
\begin{proof}
  For a square $S$, Proposition \ref{prop: UC for rhombus} gives unformity constant $2\sqrt{2}$. We consider mapping $f(x,y)=(ax,by)$. Now $f$ is $(a/b)$-bilipschitz and $f(S) = R$. By Lemma~\ref{lem:Vuo88 lemma} we obtain $2\sqrt{2} = A_S \le L^4 A_R = (a/b)^4$.
\end{proof}

\begin{rem}
  Corollary \ref{UC for rectangular} improves the lower bound introduced in \cite[5.44]{Lin05}.  
\end{rem}

\begin{thm}\label{UC in triangle}
  If $T$ is a triangle with angle $\alpha$, $\beta$ and $\gamma$ such that $\alpha \le \beta \le \gamma$. Then
  \[
    A_T \ge \frac{1}{\sin \frac{\alpha}{2}} + \frac{1}{\sin \frac{\beta}{2}}.
  \]
\end{thm}
\begin{proof}
  The medial axis of $T$ consists of subarcs of the bisectors of the triangle $T$ and it divides $T$ into three subtriangles $T_\alpha$, $T_\beta$ and $T_\gamma$. For each $m \in \{ \alpha,\beta,\gamma \}$ the triangle $T_m$ is opposite to the angle $m$. Let us choose points $x$ and $y$ from the medial axis of $T$ so that $x$ lies on the bisector of $\alpha$ and $y$ lies on the bisector of $\beta$, see Figure \ref{fig:geodesic in T}.
  \begin{figure}[!ht]\centering
    \includegraphics[scale=1]{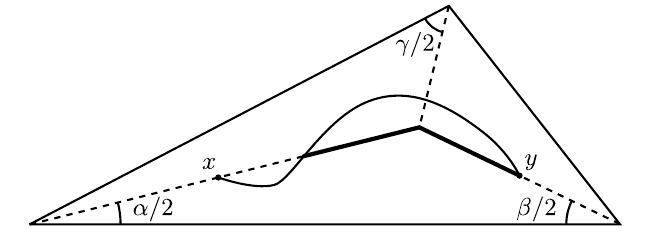}
    \caption{A curve $\Gamma$ joining $x$ and $y$ in $T$. If a part of $\Gamma$ goes outside the lower subtriangle, it can be replaced by a part of medial axis.}\label{fig:geodesic in T}
  \end{figure}
  
  The quasihyperbolic geodesic $\Gamma$ from $x$ to $y$ has to be contained in $T_\gamma$, because otherwise we could shorten the quasihyperbolic length of $\Gamma$ by replacing the part that is outside $T_\gamma$ by a part of the medial axis, see Figure \ref{fig:geodesic in T}.
  
  For $m \in \{ \alpha,\beta \}$ we denote the line segment that is a part of medial axis and starts from angle $m$ by $l_m$. We can see that if $\Gamma$ leaves from one side of $T_\gamma$, let say $l_\alpha$, it cannot come back to it, as otherwise the part could be replaced again with a line segment that is a subarc of $l_\alpha$. Thus we now that $\Gamma$ consists of three parts: $\Gamma_1$ is in $l_\alpha$, $\Gamma_2$ is in the interior of $T_\gamma$ and $\Gamma_3$ is in $l_\beta$. Here $\Gamma_1$ and $\Gamma_3$ may consists only from a single point. Note also that $\Gamma_2$ is determined by only one side of $T$ and thus it is a circular arc, because in half-plane quasihyperbolic geodesics agree with hyperbolic geodesics.
  
  Let us fix two vertices of $T$: the vertex at angle $\alpha$ is 0 and the vertex at angle $\beta$ is 1, see Figure \ref{QH geodesin in T2}. By \cite{Mar85} quasihyperbolic geodesics are smooth curves and we can observe that the radius of $\Gamma_2$ is
  \[
    r = \frac{\sin \tfrac\alpha2 \sin \tfrac\beta2 }{\sin \tfrac\alpha2 + \sin \tfrac\beta2}.
  \]
  \begin{figure}[!ht]\centering
    \includegraphics[scale=1]{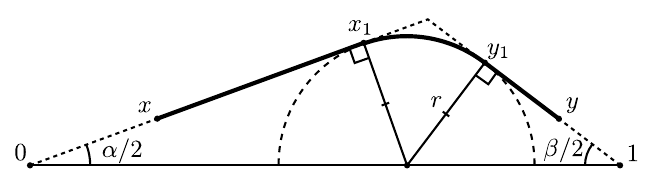}
    \caption{The quasihyperbolic geodesic $\Gamma$ joining $x$ and $y$ consists of three parts.}\label{QH geodesin in T2}
  \end{figure}
  
  We denote $\{ x_1 \} = \Gamma_1 \cap \Gamma_2$ and $\{ y_1 \} = \Gamma_2 \cap \Gamma_3$. Now $|x_1| = r/\tan \tfrac\alpha2$, $|y_1-1| = r/\tan \tfrac\beta2$ and
  \begin{eqnarray*}
    k_T(x,x_1) & = & \int_{|x|}^{|x_1|} \frac{dt}{t \sin \tfrac\alpha2} = \frac{\log |x_1|-\log |x|}{\sin \tfrac\alpha2} = \frac{\log r-\log (\tan \tfrac\alpha2)-\log |x|}{\sin \tfrac\alpha2},\\
    k_T(y,y_1) & = & \frac{\log r-\log (\tan \tfrac\beta2)-\log |y-1|}{\sin \tfrac\beta2}.
  \end{eqnarray*}
  
  Next we add a new condition for the points $x$ and $y$. We want that neither $\Gamma_1$ nor $\Gamma_3$ consist of a single point and thus we require that $|x| = |y-1| = \varepsilon$ for small enough $\varepsilon$.
  
  Now
  \[
    k_T(x,y) > k_T(x,x_1) + k_T(y_1,y) = C + \left( \frac{1}{\sin \tfrac\alpha2} + \frac{1}{\sin \tfrac\beta2} \right) (-\log |x|),
  \]
  where
  \[
    C = \frac{\log r-\log (\tan \tfrac\alpha2)}{\sin \tfrac\alpha2} + \frac{\log r-\log (\tan \tfrac\beta2)}{\sin \tfrac\beta2}
  \]
  does not depend on $|x|$.
  
  Denote $y' \in l_\beta$ be the point with $d( x , \partial T) = d( y' , \partial T)$. By assumption $\beta < \pi/2$ and thus
  \[
    |x-y| \le |x-y'| = 1-|x| \frac{\sin \tfrac{\alpha+\beta}{2}}{\sin \tfrac\beta2} = 1-D |x|,
  \]
  where $D$ is a constant not depending on $|x|$. Since $|x| = |y-1|$ and $\alpha \le \beta$ we have $d(x, \partial T) \le d(y, \partial T)$ and
  \begin{eqnarray*}
    j_T(x,y) & \le & \log \left( 1+\frac{1-D |x|}{|x|\sin \tfrac\alpha2} \right) \le \log \left( 1+\frac{1}{|x|\sin \tfrac\alpha2} \right)\\
    & \le & \log \left( \left( 1+\frac{1}{\sin \tfrac\alpha2} \right) \frac{1}{|x|} \right) = E-\log |x|,
  \end{eqnarray*}
  where $E$ does not depend on $|x|$.
  
  Putting the estimates together give us
  \begin{eqnarray*}
    \frac{k_T(x,y)}{j_T(x,y)} & \ge & \frac{C + \Big( \frac{1}{\sin \tfrac\alpha2} + \frac{1}{\sin \tfrac\beta2} \Big) (-\log |x|)}{E-\log |x|}\\
    & \ge & \lim_{|x| \to 0} \frac{C + \Big( \frac{1}{\sin \tfrac\alpha2} + \frac{1}{\sin \tfrac\beta2} \Big) (-\log |x|)}{E-\log |x|} = \frac{1}{\sin \tfrac\alpha2} + \frac{1}{\sin \tfrac\beta2}
  \end{eqnarray*}
  and the assertion follows.
\end{proof}

\begin{rem}
  Theorem \ref{UC in triangle} gives a lower bound for the uniformity constant of a triangle $T$. An upper bound
  \[
    A_T \le \frac{1}{\cos \tfrac\gamma2} \left( 2+ \frac{1}{\sin \tfrac\alpha2} + \frac{1}{\sin \tfrac\beta2} \right)
  \]
  is given \cite[5.38]{Lin05}.  
\end{rem}

Next we prove a lower bound for the uniformity constant in an ellipse and in the complement of the unit ball.

\begin{thm}\label{thm:UC for ellipse}
  Let $E \subset \R$ be an ellipse and let the ratio of the major and the minor axes be $c \ge 1$. Then
  \[
    \max \left\{ 2,\frac{2 \sqrt{c^2-1} \arcsin \sqrt{1-\frac{1}{c^2}}}{\log (2c^2-1)} \right\} \le A_E \le 2 c^4.
  \]
\end{thm}
\begin{proof}
  We denote the major axis of $E$ by $a$ and the minor axis by $b$. Now $c=a/b$.

  Let us prove first the upper bound. We choose $G=\B$ and $f(x,y)=(ax,by)$. Now $f$ is $c$-bilipschitz and $f(\B) = E$. By Lemma~\ref{lem:Vuo88 lemma} we have $A_E \le c^4 A_{\B} = 2c^4$.
  
  We prove next the lower bound. We consider points $x,y \in E \cap \mathbb{R}$ and choose $-a < x < y < a$. By symmetry and convexity of $E$ it is clear that the quasihyperbolic geodesic from $x$ to $y$ in $E$ is the line segment $[x,y]$. By choosing $x = b^2/a-a$ and $y = a-b^2/a$ gives
  \[
    k_E(x,y) = 2 \int_0^{a-b^2/a}\frac{dt}{b\sqrt{1-t^2/(a^2-b^2)}} = 2 \frac{\sqrt{a^2-b^2}}{b} \arcsin \frac{\sqrt{a^2-b^2}}{a}
  \]
  and $j_E(x,y) = \log(2(a/b)^2-1)$ implying
  \[
    A_E \ge \frac{2 \sqrt{c^2-1} \arcsin \sqrt{1-\frac{1}{c^2}}}{\log (2c^2-1)}.
  \]
  
  Let $x=si$ and $y=-si$ for $s \in (0,b)$. Now
  \[
    k_E(x,y) = 2 \int_0^s \frac{dt}{b-t} = 2(\log b-\log(b-s))
  \]
  and
  \[
    j_E(x,y) = \log \left( 1+\frac{2s}{b-s} \right) =\log(b+s)-\log(b-s).
  \]
  Because
  \[
    A_E \ge \lim_{s \to b} \frac{k_E(x,y)}{j_E(x,y)} = \lim_{s \to b} \frac{2/(b-s)}{1/(b+s)+1/(b-s)} =2,
  \]
  the assertion follows.
\end{proof}

\begin{prop}
  Domain $G=\Rn \setminus \overline{\Bn}$ is uniform and
  \[
    \frac{\pi}{\log 3} \le A_G \le \frac{4\pi}{\log 3}.
  \]
\end{prop}
\begin{proof}
  We prove first the lower bound. Let $x \in G$ and choose $y=-x$. Now by \eqref{martin-osgood formula}
  \[
    \frac{k_G(x,y)}{j_G(x,y)} \ge \frac{k_{\Rn \setminus \{ 0 \}}(x,y)}{j_G(x,y)} = \frac{\pi}{\log \left( 1+\frac{2|x|}{|x|-1} \right)} \ge \lim_{|x| \to \infty} \frac{\pi}{\log \left( 1+\frac{2|x|}{|x|-1} \right)} = \frac{\pi}{\log 3}
  \]
  and thus $A_G \ge \pi / ( \log 3 )$.
  
  Next we consider the upper bound. We use the following 4 results:

  \begin{enumerate}
  
  \item[(i)] The domain $\Bn \setminus \{ 0 \}$ is uniform with $A_{\Bn \setminus \{ 0 \}}=\tfrac{\pi}{\log 3}$. \cite[Theorem 1.9]{Lin05}
  
  \item[(ii)] For any open set $G \subset \Rn$ and for all $x,y \in G$ we have
  \[
    j_G(x,y) \le \delta_G(x,y) \le 2 j_G(x,y),
  \]
  where $\delta_G$ is the Seittenranta metric in $G$. \cite[Theorem 3.4]{Sei99}
  
  \item[(iii)] For any domain $G \subset \Rn$ and for all $x,y \in G$ we have
  \[
    k_G(x,y) \le \rho_G(x,y) \le 2 k_G(x,y),
  \]
  where $\rho_G$ is the Ferrand metric in $G$. \cite{Fer88}
  
  \item[(iv)] Metrics $\delta_G$ and $\rho_G$ are M\"obius invariant.

  \end{enumerate} 
  
  Let us fix points $x,y \in G$.  We denote M\"obius mapping $f(z) = 1/z$ and observe that $f(G) = \Bn \setminus \{ 0 \}$. By (ii) and (iii) we have
  \[
    \frac{k_G(x,y)}{j_G(x,y)} \le 2 \frac{\rho_G(x,y)}{\delta_G(x,y)} \quad \textrm{and} \quad \frac{\rho_{f(G)}(f(x),f(y))}{\delta_{f(G)}(f(x),f(y))} \le 2 \frac{k_{f(G)}(f(x),f(y))}{j_{f(G)}(f(x),f(y))}.
  \]
  These inequalities together with (i) and (iv) give
  \begin{eqnarray*}
    \frac{k_G(x,y)}{j_G(x,y)} & \le & 2 \frac{\rho_G(x,y)}{\delta_G(x,y)} = 2 \frac{\rho_{f(G)}(f(x),f(y))}{\delta_{f(G)}(f(x),f(y))} \le 4 \frac{k_{f(G)}(f(x),f(y))}{j_{f(G)}(f(x),f(y))}\\
    & \le & A_{f(G)} = \frac{4\pi}{\log 3}
  \end{eqnarray*}
  and the assertion follows.
\end{proof}

Our final uniformity constant estimate considers twice punctured space and it is not in any connection with the Ptolemy constant as the boundary of the domain is clearly not a Jordan curve.

\begin{prop}\label{thm: UC for twice punctured plane}
  For $G = \R \setminus \{ -e_1,e_1 \}$ we have 
  \[
    A_G \ge \frac{2 \arcsinh \beta}{\log \left( 1+\frac{2\beta}{\sqrt{1+\beta^2}} \right)} \approx 3.5131,
  \]
  where $\beta \approx 3.1841$ is the solution of $\arcsinh t+\arctan t = \pi$ for $t>0$.
\end{prop}
\begin{proof}
  We consider $k_G(t e_2,-t e_2)$ for $t>0$ and $G = \R \setminus \{ -e_1,e_1 \}$. By \cite{MarOsg86} and \cite{Vai09} we know the geodesics joining any two points in $G$. For small $t$ the geodesic segment between $x=t e_2$ and $-x$ is the line segment $[x,-x]$ and for large $t$ there exists more than one geodesic segments joining $x$ and $-x$. In this case the geodesics are circular arcs with center at $-e_1$ or $e_1$. There is also a value of $t$ such that geodesics joining $x$ and $-x$ are circular arcs and the line segment $[-x,x]$. We show that this limiting value of $t$ is $\beta$.

  We find formula for quasihyperbolic length of line segment $[x,-x]$. By definition
  \[
    \ell_{k_G} ([t e_2,-t e_2]) = 2 \ell_{k_G} ([0,t e_2]) = 2 \int_0^t \frac{dz}{\sqrt{1+z^2}} = 2 \arcsinh t.
  \]

  Next we find formula for the quasihyperbolic length for the longer circular arc $C(x,-x)$ with center $e_1$ and joining $x$ and $-x$. By definition
  \[
    \ell_{k_G} (C(x,-x)) = 2\pi- \measuredangle (x,e_1,-x) = 2(\pi -\arctan t).
  \]

  Now it is clear that
  \begin{eqnarray*}
    k_G(x,-x) & = & \ell_{k_G} ([x,-x]) \wedge \ell_{k_G} (C(x,-x))\\
    & = & 2 \min \{ \arcsinh t,\pi -\arctan t \}
  \end{eqnarray*}
  and $\arcsinh t = (\pi -\arctan t)$ is equivalent to $t=\beta$.

  We next show that the solution $\beta$ is unique. Consider the function $f(t) = \arcsinh t - (\pi -\arctan t)$. Since $f'(x) = (1+\sqrt{1+t^2})/(1+t^2) > 0$ the function $f(t)$ is strictly increasing and hence $\beta$ is a unique solution. Since the functions $\arcsinh t$ and $\pi -\arctan t$ are strictly monotone, it is clear that $k_G(x,-x)$ obtains its maximum at $\beta$.

  By definition
  \[
    j_G(x,-x) = \log \left( 1+ \frac{2t}{\sqrt{1+t^2}} \right)
  \]
  and thus
  \[
    \sup_{y,z \in G} \frac{k_G(y,z)}{j_G(y,z)} \ge \frac{k_G(\beta e_1,-\beta e_1)}{j_G(\beta e_1,-\beta e_1)} = \frac{2 \arcsinh \beta}{\log \left( 1+\frac{2\beta}{\sqrt{1+\beta^2}} \right)}
  \]
  and the assertion follows.
\end{proof}

\textbf{Acknowledgements.} The authors thank M. Vuorinen for the introduction to the topic and the useful comments on the manuscript.


\addresseshere

\newpage
\section{Appendix}

In this Appendix we provide proofs for results intoduced in Section \emph{Preliminary results}.

\begin{proof}[Proof of Lemma~\ref{prop:circumcenter}]
  The circumcenter is the intersection of perpendicular bisectors of line segments $[a,b]$ and $[b,c]$. Thus we have $(a+b)/2+i(b-a)s=(b+c)/2+i(c-b)t$ and $(\overline{a}+\overline{b})/2+i(\overline{b}-\overline{a})s=(\overline{b}+\overline{c})/2+i(\overline{c}-\overline{b})t$ which give
  \[
    s=\frac{i}{2} \frac{\overline{a}(b-c)+\overline{b}(a-c)+\overline{c}(2c-a-b)}{\overline{a}(b-c)+\overline{b}(c-a)+\overline{c}(a-b)}
  \]
  and the circumcenter is
  \[
    \frac{a+b}{2}+i(b-a)s = \frac{(b-c)|a|^2+(c-a)|b|^2+(a-b)|c|^2}{(b-c)\bar{a}+(c-a)\bar{b}+(a-b)\bar{c}}. \qedhere
  \]
\end{proof}

\begin{proof}[Proof of Lemma~\ref{lem:circle_outer_angle}]
  Let $\alpha + \gamma \le \beta + \delta$. Now $\alpha + \gamma \le \pi \le \beta + \delta$ and if $\alpha + \gamma = \pi = \beta + \delta$, then the assertion follows.
  
  We assume $\alpha + \gamma < \beta + \delta$. Now $\mathcal{C}_{A}$ and $\mathcal{C}_{C}$ intersect at points $b$ and $d$. Because $\alpha + \gamma < \pi$, the circular arcs corresponding the angles $\alpha$ and $\gamma$ are subarcs of a semicircle, see Figure \ref{fig:outer_angle}.

\begin{figure}[!ht]\centering
\includegraphics[width=\textwidth]{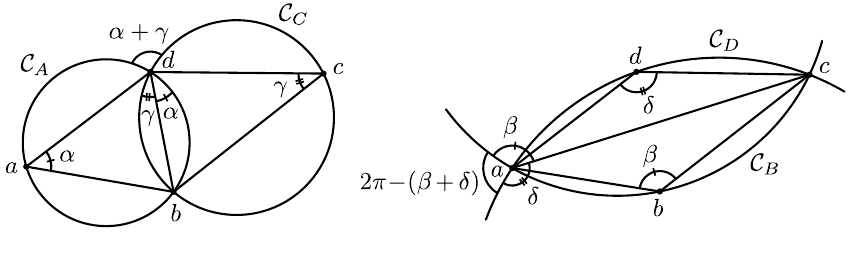}
\caption{Left: Circles $\mathcal{C}_{A} = \mathcal{C}(a,b,d)$ and $\mathcal{C}_{C} = \mathcal{C}(b,c,d)$ in Lemma~\ref{lem:circle_outer_angle}. Right: Circles $\mathcal{C}_{B} = \mathcal{C}(a,b,c)$ ja $\mathcal{C}_{D} = \mathcal{C}(c,d,a)$ in Lemma~\ref{lem:circle_outer_angle}.\label{fig:outer_angle}}
\end{figure}

The angle between circle $\mathcal{C}_{A}$ and line $\ell (b, d)$ is $\alpha$, and the angle between circle $\mathcal{C}_{C}$ and line $\ell (b, d)$ is $\gamma$. Thus the outer angle between circles $\mathcal{C}_{A}$ and $\mathcal{C}_{C}$ is $\alpha + \gamma$.

The situation for circles $\mathcal{C}_{B}$ and $\mathcal{C}_{D}$ is represented in Figure \ref{fig:outer_angle}. Similar computation as above gives that the angle between circles $\mathcal{C}_{B}$ and $\mathcal{C}$ is $2\pi - (\beta + \delta) = \alpha + \gamma$.
\end{proof}

\begin{proof}[Proof of Lemma~\ref{lem:convex_polygon_Stheta}]
  We may assume $\alpha + \gamma \leq \pi \leq \beta + \delta = 2\pi - (\alpha + \gamma)$. By Lemma~\ref{lem:circle_outer_angle} the angle between circles $\mathcal{C}_{B} = \mathcal{C}(a,b,c)$ and $\mathcal{C}_{D} = \mathcal{C}(c,d,a)$ is $\alpha + \gamma = 2\pi - (\beta + \delta)$. Thus we know that there exists a M\"obius transformation $m$ with $m(a)=0$, $m(b)=1$ and $m(c)=\infty$, and which maps the points $a, b, c, d$ in the same order to the curve $S_\theta$ $J_{\alpha + \gamma}$. Both circles $\mathcal{C}_{B}$ and $\mathcal{C}_{D}$ are mapped to a line and the interior of $\mathcal{C}_{B}$ is mapped to the upper half space, see Figure \ref{fig:convex_polygon_Stheta}.
\end{proof}

\begin{figure}[!ht]\centering
\includegraphics[width=\textwidth]{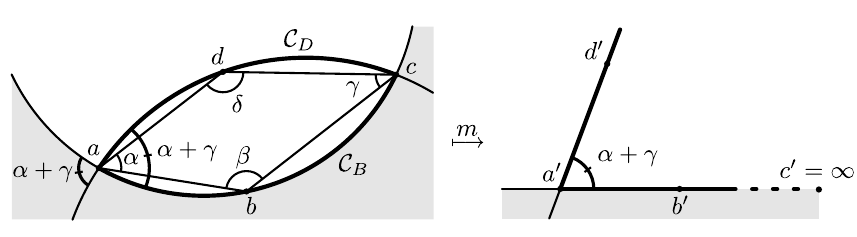}
\caption{The mapping $m$ in Lemma~\ref{lem:convex_polygon_Stheta}. \label{fig:convex_polygon_Stheta}}
\end{figure}

\begin{proof}[Proof of Lemma~\ref{lem:simplequadrilateral}]
  If $(a,b,c,d)$ is convex, then the assertion follows from Lemma~\ref{lem:convex_polygon_Stheta}.
  
  If $(a,b,c,d)$ is not convex, then we may assume that $\alpha + \gamma \leq \pi \leq \beta + \delta$ and $\beta > \pi$. We show that there exists a domain $D$ such that the points $a,b,c,d$ are in the same order in $\partial D$ and $\partial D$ consists of two circular arcs with angle $\alpha+\gamma$. Now $m$ can be found as in the proof of Lemma~\ref{lem:convex_polygon_Stheta}.
  
  Let $\mathcal{C}_{B} = \mathcal{C}(a,b,c)$ and $\mathcal{C}_{D} = \mathcal{C}(c,d,a)$. We choose $D$ so that $\partial D \subset \mathcal{C}_{B} \cup \mathcal{C}_{D}$ and $a,b,c,d \in \partial D$, see Figure~\ref{fig:ymp_konk_kulma}.

\begin{figure}[!ht]\centering
\includegraphics[scale=1]{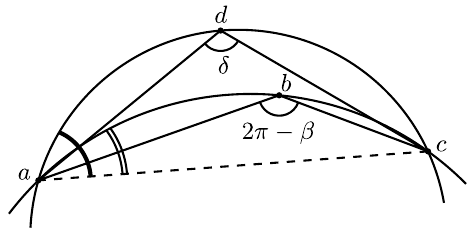}
\caption{Proof of Lemma~\ref{lem:simplequadrilateral}. Thick line represents angle $\pi - \delta$ and double line angle $\beta - \pi$. \label{fig:ymp_konk_kulma}}
\end{figure}

  Now $\measuredangle(a,b,c) = 2\pi-\beta$ and thus $\measuredangle(c,a,d) = \pi-(2\pi-\beta) = \beta-\pi$. Similarly, $\measuredangle(a,d,c) = \delta$ and the angle between $\mathcal{C}_{D}$ and line $\ell (a,c)$ is $\pi-\delta$. The angle between $\mathcal{C}_{B}$ and $\mathcal{C}_{D}$ is now
  \[
    (\pi - \delta) - (\beta - \pi) = 2\pi - (\beta + \delta) = \alpha + \gamma
  \]
  and the assertion follows.
\end{proof}

\begin{proof}[Proof of Lemma~\ref{prop:quadrilateral}]
  By Lemma~\ref{lem:simplequadrilateral} we can map a simple quadrilateral $abcd$ by a M\"obius transformation $m_1$ to $S_\theta$ in a way that $m_1(a)=0$, $m_1(b)=1$, $m_1(c)=\infty$ and $m_1(d)=te^{i\theta}$, where $t>0$ and $\theta=(\alpha+\gamma)/2$. Similarly there is a M\"obius transformation $m_2$ which takes the quadrilateral $m(a)m(b)m(c)m(d)$ to $S_{\theta}$ with $m_2(m(a))=0$, $m_2(m(b))=1$, $m_2(m(c))=\infty$ and $m_1(m(d))=se^{i\theta}$. 
  
  Let us consider parallelogram $m(a)=0$, $m(b)=r$, $m(c)=r+ue^{i\theta/2}$ and $m(d)=ue^{i\theta/2}$ for $r,u>0$. Since cross ratio is invariant under M\"obius transformation we obtain
  \[
    [m(a),m(b),m(c),m(d)] = 1-\left( \frac{u}{r} \right)^2 e^{i\theta}  \quad \textrm{and} \quad [0,1,\infty,s e^{i\theta}] = 1-s e^{i\theta}.
  \]
  Choosing $s=u/r >0$ we obtain $m=m_2^{-1} \circ m_1$ and the assertion follows.
\end{proof}

\begin{proof}[Proof of Lemma~\ref{lem:visualanglemetricH2}]
  The largest value follows from \cite[section 3.3]{KleLinVuoWan14} and the smallest value is clear as then $|\measuredangle (x,z,y)|=0$.
\end{proof}

\begin{proof}[Proof of Lemma~\ref{lem:kompleksilukutulos}]
  The claim is equivalent to
  \[
    \frac{\arctan \left( c\cdot y/x \right)}{c\cdot y/x} \; > \; \frac{\arctan y/x}{y/x}
  \]
  and we prove this inequality by showing that the function $f(t) = (\arctan t)/t$ is strictly decreasing and $f(t) \to 1$ as $t \to 0$.
  
  By differentiation we obtain
  \[
    f'(t) = \frac{t-(1+t^2)\arctan t}{(1+t^2)t^2} = \frac{g(t)}{(1+t^2)t^2},
  \]
  where function $g'(t) = -2t\arctan t$ is negative for $t>0$.
  
  The limit $f(t) \to 1$ as $t \to 0$ is obtained by l'Hôpital's rule.
\end{proof}

\begin{proof}[Proof of Lemma~\ref{lem:apulause_4.12}]
  By differentiation $(x_0/a)^2+(y_0/b)^2 = 1$ gives $2x/a^2+2yy'/b^2=0$, which implies $y'=-xb^2/(ya^2)$. The normal of the tangent of $\partial E$ at point $(x_0,y_0)$ is
  \[
    y-y_0=\frac{y_0a^2}{x_0 b^2}(x-x_0)
  \]
  and since $|x|<a$ the normal at $(x_0,y_0)$ intersects the real axis at $(x_0(1-b^2/a^2),0)$.
  
  The maximal disk $B^2(z,r)$ contained in $E$ intersects $\partial E$ at $(x_0,y_0)$, where $x_0=t/(1-b^2/a^2))$, whenever $|t|/(1-b^2/a^2) \le a$, which is equivalent to $|t|\le a-b^2/a$. The first part of the assertion follows, because now
  \[
    y_0 = \pm b\sqrt{ \left( \frac{at}{a^2-b^2} \right)^2 -1 }
  \]
  and
  \[
    d(z,\partial E) = \sqrt{(t-x_0)^2+y_0^2} = b \sqrt{1-\frac{t^2}{a^2-b^2}}.
  \]
  
  Let us next consider the case $a-b^2/a < |t| < a$. The curvature of $\partial E$ at $(\pm a,0)$ is $b^2/a$. If $a-b^2/a < |t| < a$, then $a-|t| < b^2/a$ and thus the maximal disk $B^2(z,r)$ contained in $E$ intersects $\partial E$ at $(\pm a,0)$. Now the assertion follows easily.
\end{proof}


\begin{thebibliography}{99}

\bibitem{Ahl63} L.V. Ahlfors: Quasiconformal reflections. Acta Math. 109 (1963), 291--301.

\bibitem{AndAnd06} T. Andreescu, D. Andrica: Complex numbers from A to... Z. Translated and revised from 2001 Romanian original. Birkh\"auser Boston, Inc., Boston, MA, 2006.

\bibitem{Bec78} M. Becker, E.L. Stark: On a hierarchy of quolynomial inequalities for $\tan x$. Univ. Beograd. Publ. Elektrotehn. Fak. Ser. Mat. Fiz. No. 602/633 (1978), 133--138.

\bibitem{Ber09} M. Berger: Geometry I. Springer-Verlag, Berlin Heidelberg, 2009.

\bibitem{Fer88} J. Ferrand: A characterization of quasiconformal mappings by the behaviour of a function of three points. Complex analysis, Joensuu 1987, 110--123, Lecture Notes in Math., 1351, Springer, Berlin, 1988.

\bibitem{Geh79} F.W. Gehring, B.G. Osgood: Uniform domains and the quasihyperbolic metric. J. Analyse Math. 36 (1979), 50–74 (1980).

\bibitem{KleHar15} P. Harjulehto, R. Klén: Examples of fractals satisfying the quasihyperbolic boundary condition. Aust. J. Math. Anal. Appl. 12 (2015), no. 1, Art. 9.

\bibitem{Har12} E. Harmaala: Ptolemaioksen lauseen yleistyksestä tasokäyrien tasaisuuden luonnehtimiseksi. Master's thesis (in Finnish), University of
Turku, 2012, https://www.doria.fi/handle/10024/77151.

\bibitem{Kle08} R. Klén: On hyperbolic type metrics. Ann. Acad. Sci. Fenn. Math. Diss., No. 152 (2009).

\bibitem{KleLinVuoWan14} R. Klén, H. Lindén, M. Vuorinen, G. Wang: The Visual Angle Metric and Möbius Transformations. Comput. Methods Funct. Theory 14 (2014), no. 2--3, 577-–608.

\bibitem{Lin05} H. Lindén: Quasihyperbolic geodesics and uniformity in elementary domains. Dissertation, University of Helsinki, Helsinki, 2005. Ann. Acad. Sci. Fenn. Math. Diss. No. 146 (2005), 50 pp.

\bibitem{Mar85} G. Martin: Quasiconformal and bi-Lipschitz homeomorphisms, uniform domains and the quasihyperbolic metric. Trans. Amer. Math. Soc. 292 (1985), 169--191.

\bibitem{MarOsg86} G.J. Martin, B.G. Osgood: The quasihyperbolic metric and associated estimates on the hyperbolic metric. J. Analyse Math. 47 (1986), 37--53.

\bibitem{PinReiSha01} Y. Pinchover, S. Reich, I. Shafrir: Problems and Solutions: Solutions: The Ptolemy Constant of a Normed Space: 10812. Amer. Math. Monthly 108 (2001), no. 5, 475--476.

\bibitem{Ric66} S. Rickman: Characterization of quasiconformal arcs.  Ann. Acad. Sci. Fenn. A1 395 (1966), 1--29.

\bibitem{Sei96} P. Seittenranta: Möbius invariant metrics and quasiconformal maps. Licentiate thesis, University of
Helsinki, 1996.

\bibitem{Sei99} P. Seittenranta: M\"obius-invariant metrics. Mathematical Proceedings of the Cambridge Philosophical Society 125 (1999), 511--533. 

\bibitem{Vai09} J. V\"ais\"al\"a: Quasihyperbolic geometry of planar domains. Ann. Acad. Sci.
Fenn. Math. 34 (2009), 447--473.

\bibitem{Val70} J.E. Valentine: An analogue of Ptolemy's theorem in spherical geometry. Amer. Math. Monthly 77 (1970), 47--51.

\bibitem{Val70B} J.E. Valentine: An analogue of Ptolemy's theorem and its converse in hyperbolic geometry. Pacific J. Math. 34 (1970), 817--825.

\bibitem{Vuo88} M. Vuorinen: Conformal geometry and quasiconformal mappings. Lecture notes in Mathematics, 1319. Springer-Verlag, Berlin, 1988.

\bibitem{Zuo12}  Z. Zuo: The Ptolemy constant of absolute normalized norms on $\R$. J. Inequal. Appl. 2012, 2012:107.

\end{thebibliography}
\end{document}